\documentclass{article}

\usepackage{amsmath}
\usepackage{mathtools}
\usepackage{amssymb}
\usepackage{amsthm}
\usepackage{amscd}
\usepackage[alphabetic,backrefs,lite]{amsrefs}
\usepackage{mathrsfs}
\usepackage[all,cmtip]{xy}
\usepackage{enumerate}
\usepackage{comment}
\usepackage{cleveref}

\newcommand{\A}{{\mathbb A}}
\newcommand{\C}{{\mathbb C}}
\newcommand{\F}{{\mathbb F}}
\newcommand{\G}{{\mathbb G}}

\newcommand{\Q}{{\mathbb Q}}
\newcommand{\R}{{\mathbb R}}
\newcommand{\DeligneS}{{\mathbb S}}

\newcommand{\Z}{{\mathbb Z}}
\newcommand{\Qbar}{{\overline{\Q}}}
\newcommand{\Zhat}{{\hat{\Z}}}
\newcommand{\kbar}{{\overline{k}}}

\newcommand{\ab}{{\operatorname{ab}}}
\newcommand{\ad}{{\operatorname{ad}}}
\newcommand{\der}{{\operatorname{der}}}

\newcommand{\opp}{{\operatorname{opp}}}
\newcommand{\nr}{{\operatorname{nr}}}

\newcommand{\isom}{\simeq}

\newcommand{\plim}[1][]{\mathop{\varprojlim}\limits_{#1}}

\DeclareMathOperator{\tr}{tr}

\DeclareMathOperator{\ch}{ch}
\DeclareMathOperator{\lcm}{lcm}
\DeclareMathOperator{\Hom}{Hom}

\DeclareMathOperator{\End}{End}
\DeclareMathOperator{\Aut}{Aut}
\DeclareMathOperator{\GL}{GL}

\DeclareMathOperator{\Lie}{Lie}
\DeclareMathOperator{\Nrd}{Nrd}
\DeclareMathOperator{\Trd}{Trd}
\DeclareMathOperator{\Frob}{Frob}
\DeclareMathOperator{\Br}{Br}
\DeclareMathOperator{\Pic}{Pic}

\newtheorem{theorem}{Theorem}[section]
\newtheorem{proposition}[theorem]{Proposition}
\newtheorem{lemma}[theorem]{Lemma}

\theoremstyle{definition}
\newtheorem{definition}[theorem]{Definition}

\newtheorem*{acknowledgements}{Acknowledgements}

\theoremstyle{remark}

\newtheorem{remark}[theorem]{Remark}

\crefname{theorem}{Theorem}{Theorems}
\crefname{proposition}{Proposition}{Propositions}
\crefname{lemma}{Lemma}{Lemmata}
\crefname{corollary}{Corollary}{Corollaries}
\crefname{conjecture}{Conjecture}{Conjectures}
\crefname{definition}{Definition}{Definitions}
\crefname{example}{Example}{Examples}
\crefname{remark}{Remark}{Remarks}
\crefname{section}{Section}{Sections}
\crefname{equation}{the equation}{the equations}
\Crefname{equation}{The equation}{The equations}

\begin{document}

\title{Rational points and Brauer--Manin obstruction on Shimura varieties of level one classifying abelian varieties with quaternionic multiplication}
\author{Koji Matsuda\thanks{The University of Tokyo}}
\date{}

\maketitle

\begin{abstract}
We show that the Shimura varieties of level one parametrizing QM-abelian varieties have rarely rational points.
\end{abstract}

\section{Introduction}

In his celebrated paper, Mazur \cite{MazurX1} showed that for a sufficiently large positive integer $N$ the modular curve $X_1(N)$, which classifies elliptic curves with their level-$\Gamma_1(N)$ structures, has no non-cuspidal $\Q$-rational points.
Since then, similar statements, which claim that modular curves of ``sufficiently large level" have no rational points over suitable number fields, have been proved by many authors, and have led to many important applications to the study of elliptic curves and to the study of Fermat's type equations such as $x^p + y^p = z^p$.

The modular curves can be considered as the Shimura varieties associated with the algebraic group $\GL_2 \Q$, which are one of the most classical and simple Shimura varieties.
Replacing the group $\GL_2 \Q$ by the algebraic group induced by the unit group of an indefinite division quaternion algebra over $\Q$, which is a twist of $\GL_2 \Q$, we obtain certain types of Shimura varieties, which are called the Shimura curves.
In 1986 Jordan \cite{Jordan}*{Theorem 6.3} proved, for a certain indefinite division quaternion algebra $B$ over the rationals, that the Shimura curve $M^B$ of level $1$ has no rational points over certain imaginary quadratic fields $k$ satisfying that $B \otimes_\Q k \isom M_2(k)$.
Skorobogatov \cite{SkorobogatovShimuraCovering}*{Theorem 3.1} proved a strengthened version of Jordan's result.
Namely, in appropriate situations, the failure of the Hasse principle of a Shimura curve is explained by the Brauer--Manin obstruction, i.e., the Brauer--Manin set $M^B(\A_k)^{\Br}$ is empty for suitable imaginary quadratic fields satisfying that $B \otimes_\Q k \isom M_2(k)$.
Various generalizations of these results are studied by many authors.
For example, Rotger--de Vera-Piquero \cite{RdVP}*{Theorem 1.1} showed it for $k$ not necessarily splitting $B$, and Arai \cite{AraiNonexistence}*{Theorem 1.1}, \cite{AraiBM}*{Theorem 2.3, 2.4} showed it for higher degree number fields $k$ not necessarily splitting $B$.

In this paper we generalize them for higher dimensional Shimura varieties.
Let $F$ be a totally real number field of degree $d$, $B$ a totally indefinite division quaternion algebra over $F$, and let $M^B$ be the PEL Shimura variety associated with $B^*$ of level $1$, which is a normal projective variety of dimension $d$ over $\Q$.
This is the coarse moduli scheme classifying $*$-polarized QM-abelian varieties.
As in the case of Shimura curves, we prove that $M^B$ has rarely rational points.

\begin{theorem}[\cref{maintheorem1}]
Let $k$ be a number field containing the normal closure of $F/\Q$.
Assume that there exist a prime $\mathfrak{q}$ of $k$ above a rational prime $q$ and a prime $p_F$ of $F$ above a rational prime $p$ satisfying all of the following conditions:
\begin{enumerate}
\item $\mathfrak{q} \nmid 2 \Delta$.
\item The absolute inertia degree of $\mathfrak{q}$ is odd.
\item $B \otimes_F F(\sqrt{-q}) \not\isom M_2(F(\sqrt{-q}))$.
\item $p_F \not\in Q(N(\mathfrak{q}), 2)$.
\item For every prime $\mathfrak{p}$ of $k$ above $p$, the greatest common divisor of $2r$ and the absolute inertia degree of $\mathfrak{p}$ is $r$, where $r$ is the absolute inertia degree of $p_F$.
\item $p_F \mid \mathfrak{d}_{B/F}$.
\end{enumerate}
Then $M^B(\A_k)^{\Br_1} = \varnothing$ and in particular $M^B(k) = \varnothing$.
\end{theorem}

For precise definitions of notations, see below.
We show this theorem by studying the canonical characters of $*$-polarized QM-abelian varieties as in \cites{Jordan,AraiBM} and by studying an etale covering of $M^B$, which is called a Shimura covering, as in \cite{SkorobogatovShimuraCovering}.

\textit{Notation.}
Unless otherwise stated, we use the following notations:
For a number field or a local field $K$ we denote the ring of integers of $K$ by $O_K$, and for a (or the) prime ideal (or nonarchimedean valuation) $\mathfrak{p}$ of $K$,
the symbol $\F_\mathfrak{p}$ denotes the residue field of $K$ at $\mathfrak{p}$.
If $K$ is a local field, we also denote the residue field by $\F_K$.
For a finite field or for a local field $K$ let $\Frob_K \in G_K$ be a (or the) Frobenius.
For a number field $K$ let $h_K$ be its class number and $h'_K$ be the exponent of the class group of $K$, and also for a prime ideal (or for a nonarchimedean valuation) $\mathfrak{p}$ of $K$ let $\Frob_\mathfrak{p} \in G_K$ be a Frobenius of $K_\mathfrak{p}$.
For an extension of global fields $L/K$ we denote the discriminant and the different by $\mathfrak{d}_{L/K}$ and $\mathfrak{D}_{L/K}$ respectively.

$F$ is a totally real number field of degree $d$,
$B$ is a totally indefinite division quaternion algebra over $F$,
$\mathfrak{d}_{B/F}$ is the product of the primes at which $B/F$ is ramified,
$\Nrd$ and $\Trd$ is the reduced norm and the reduced trace of $B/F$ respectively,
$b \mapsto b^*$ is a positive involution on $B$,
i.e., an involution on $B$ satisfying for all $b \in B$
that $\tr_{B/\Q} (bb^*) > 0$.
In this case by the theorem of Albert (for example see \cite{MumfordAV}*{Chapter IV, Section 21, Theorem 2}) we have that
$F$ is fixed by $*$ and that $(B, *) \otimes_\Q \R \isom (M_2(\R), t)^d$ as $\R$-algebras with involutions, where $t$ is the transpose of matrices.
We fix this isomorphism throughout this paper.
$O_B$ is a maximal order of $B$ which is invariant under the involution $*$.
Note that obviously we have $O_B \cap F = O_F.$
$\operatorname{Ram}(B/F)$ is the set of ramified primes of $B/F$.
$\Delta$ is the product of the rational prime numbers $p$ at which $B$ is ramified, i.e., the rational prime numbers $p$ such that $F$ is ramified at $p$ or that $B$ is ramified at a place of $F$ above $p$.
Also $\Delta'$ is the product of the rational prime numbers $p$ satisfying that $B$ is ramified at a place of $F$ above $p$.
Throughout this paper, for each prime $\mathfrak{p}$ of $F$ which splits $B$, we fix an isomorphism from $O_B \otimes_{O_F} O_{F,\mathfrak{p}}$ to $M_2(O_{F,\mathfrak{p}})$, and identify them.

For a group (or for a group scheme) $G$, the symbol $Z(G), G^\ad$, $G^\der$, and $G^\ab$ denotes the center, the quotient by its center, the commutator subgroup, and the maximal abelian quotient of $G$ respectively.

For an algebra $R$ and its subset $S$, the symbol $C_R(S)$ denotes the centralizer of $S$ in $R$.

We denote the ring of adeles (respectively the ring of finite adeles) over the rationals $\Q$
by $\A$ (respectively by $\A_f$).
For an abelian variety $A$ over a field $k$ and for a prime $\ell$, we denote $V_\ell A = T_\ell A \otimes_{\Z_\ell} \Q_\ell$, $\hat{T}A := \plim A[m](\kbar) = \prod_p T_p A,$ and similar for $\hat{V}A$.

All rings have an identity element, and all ring homomorphisms take the identity element to the identity element.

We denote the Deligne torus, i.e. the Weil restriction of $\G_m$ from $\C$ to $\R$, by $\DeligneS$.

For a field $k$ and for a separated $k$-scheme $X$ of finite type, we denote the cohomological Brauer group $H^2(X, \G_m)$ by $\Br X$ and simply call it the Brauer group.
Let $\Br_0 X$ be the image of $\Br k \to \Br X$ induced by the structure map $X \to k$, and let $\Br_1 X$ be the kernel of the canonical map $\Br X \to \Br X_\kbar$.
For a $k$-group $S$ of multiplicative type (i.e., a group $k$-scheme which is a group subscheme of $\G_{m,k}^n$ etale locally for some $n$) and for a map $\lambda \colon \hat{S} := \Hom_\kbar(S(\kbar), \G_m(\kbar)) \to \Pic X_\kbar$ of Galois modules, let the symbol $\Br_\lambda X$ denote the inverse image of $\lambda(H^1(k, \hat{S})) \subseteq H^1(k, \Pic X_\kbar)$ under the map $\Br_1 X \to H^1(k, \Pic X_\kbar)$, which is induced by the Hochschild--Serre spectral sequence
\begin{equation*}
H^p(k, H^q(X_\kbar, \G_m)) \Rightarrow H^n(X, \G_m).
\end{equation*}

For a number field $k$, for a proper $k$-scheme $X$, and for a subset $S$ of $\Br X$, we denote the set of the adelic points which are orthogonal to $S$ with respect to the Brauer--Manin pairing $X(\A_k) \times \Br X \to \Q/\Z$ by $X(\A_k)^S$.
When $S = \Br X, \Br_1 X, \Br_0 X,$ and $\Br_\lambda X$, we simply write $X(\A_k)^S$ by $X(\A_k)^{\Br}, X(\A_k)^{\Br_1}$, and so on respectively.

\section{The Shimura coverings of Shimura varieties parametrizing QM-abelian varieties} \label{Section:Shimura_covering}

First we recall basic facts about QM-abelian varieties and about our Shimura varieties.

For an abelian scheme $A$ over a scheme $S$,
the symbol $\Lie(A/S)$ denotes the pullback of the tangent sheaf $\mathscr{T}_{A/S} = (\Omega_{A/S}^1)^\vee$ along with the $0$-section,
which is a locally free $\mathscr{O}_S$-module of rank $\dim A/S$.
On $\Lie(A/S)$ the ring $\End(A/S)$ acts from left canonically.
Let $A^\vee$ denote the dual abelian scheme of $A$ over $S$.
Note that the dual abelian scheme does always exist as a scheme even if the abelian scheme is not projective,
see \cite{FaltingsChai}*{Chapter I, \S 1}.

\begin{definition} \label{def:QMAV}
Let $S$ be a scheme.
\begin{enumerate}
\item A QM-abelian scheme by $O_B$ over $S$ is a pair $(A, i)$, where $A$ is an abelian scheme over $S$ and
$i \colon O_B \to \End(A/S)$ is a ring homomorphism,
such that $\Lie(A/S)$ is locally free of rank $2$ over $O_F \otimes_\Z \mathscr{O}_S$.

\item Let $A/S$ be an abelian scheme and $i \colon O_B \to \End(A/S)$ a ring homomorphism.
A $*$-polarization on $(A,i)$ by $O_B$ over $S$ or, to be more precise,
a $*$-polarization on $(A,i)$ of type $(O_B, b \mapsto b^*)$ over $S$ is a polarization over $S$
which commutes with the action of $O_B$, i.e., a polarization $\lambda$ of $A$ such that for every $b \in O_B$ the diagram
\begin{equation*}
\begin{aligned}
\xymatrix{
A \ar[r]^\lambda \ar[d]_{i(b^*)} & A^\vee \ar[d]^{i(b)^\vee} \\
A \ar[r]^\lambda & A^\vee
}
\end{aligned}
\end{equation*}
is commutative.

\item A $*$-polarized QM-abelian scheme by $O_B$ or, to be more precise,
a $*$-polarized QM-abelian scheme of type $(O_B, b \mapsto b^*)$ over $S$ is a pair $(A, i , \lambda)$,
where $(A,i)$ is a QM-abelian variety and $\lambda$ is a $*$-polarization on it.
\end{enumerate}
We often simply call them a QM-abelian scheme and a $*$-polarized QM-abelian scheme,
and also write $A$ instead of $(A,i)$ and $(A,i,\lambda)$ when no confusion can arise.
\end{definition}

For the definition of linear-algebraic objects, see \cite{Lan}*{Section 1.1.4}.

Let $\Lambda$ be a free left $O_B$-module of rank one and $\psi$ a nondegenerate alternating $(O_B, *)$-pairing on $\Lambda$, i.e., a nondegenerate alternating $\Z$-bilinear pairing $\psi \colon \Lambda \times \Lambda \to \Z$ satisfying that $\psi(bx,y) = \psi(x,b^*y)$ for all $b \in O_B$ and for all $x, y \in \Lambda$.
(Such a form does always exist by \cite{Milne}*{Lemma 1.1}.)
We fix them in the rest of the paper.
Define a subgroup scheme $G_1$ of $\operatorname{Res}_{O_F/\Z} \GL_{O_B} \Lambda \times \operatorname{Res}_{O_F/\Z} \G_{m,O_F}$ over $\Z$ by
\begin{equation*}
G_1(R) = \{ (g, \nu) \in \GL_{O_B \otimes_\Z R} (\Lambda \otimes_\Z R) \times (O_F \otimes_\Z R)^* | \psi(g-,g-) = \psi(-, \nu -) \}
\end{equation*}
for a ring $R$, and let $G_0 = \operatorname{SymAut}_{O_B}(\Lambda)$, i.e., the group scheme over $\Z$ satisfying for a ring $R$ that $G_0(R) = \operatorname{SymAut}_{O_B \otimes_\Z R}(\Lambda \otimes R)$, the set of symplectic automorphisms of the symplectic $O_B \otimes_\Z R$-module $(\Lambda \otimes R, \psi)$.
These group schemes fit into the following Cartesian diagram:
\begin{equation*}
\begin{aligned}
\xymatrix{
G_0 \ar[rr]^-{\text{inclusion}} \ar[d]^{\nu} & & G_1 \ar[d]^{\nu} \\
\G_m \ar[rr]^-{\text{canonical}} & & \operatorname{Res}_{O_F/\Z} \G_{m,O_F}.
}
\end{aligned}
\end{equation*}
Then the canonical projection gives the isomorphisms
\begin{equation*}
\begin{aligned}
G_1(R) & \isom \{ g \in \GL_{O_B \otimes R} (\Lambda \otimes R) | \Nrd(g) \in (O_F \otimes R)^* \}, \\
G_0(R) & \isom \{ g \in \GL_{O_B \otimes R} (\Lambda \otimes R) | \Nrd(g) \in R^* \},
\end{aligned}
\end{equation*}
and hence
\begin{equation*}
\begin{aligned}
Z_1(R) := Z(G_1)(R) & \isom (O_F \otimes R)^*, \\
Z_0(R) := Z(G_0)(R) & \isom \{ g \in (O_F \otimes R)^* | g^2 \in R^* \},
\end{aligned}
\end{equation*}
for a ring $R$ which is flat over $\Z$ or in which $[\Lambda^\vee : \Lambda]\Delta'$ is invertible.
(Here, if $R$ is flat over $\Z$, for $g \in \GL_{O_B \otimes R} (\Lambda \otimes R)$ we define $\Nrd(g)$ as the reduced norm of $g$ as an element of $\GL_{O_B \otimes R_\Q} (\Lambda \otimes R_\Q)$.)
In particular over the ring $\Z[1/ [\Lambda^\vee : \Lambda]\Delta']$ we have the isomorphism $G_1 \isom \operatorname{Res}_{O_F/\Z} \GL_{O_B}\Lambda$.

By \cite{MilneSh}*{Proposition 8.14}, which cites \cite{Zink}*{Lemma 3.1}, there exists a canonical map $h \colon \DeligneS \to G_{0,\R}$ such that letting $X_0$ be the $G_0(\R)$-conjugacy class of $h$, the pair $(G_{0,\Q}, X_0)$ gives the PEL Shimura datum associated with our $(B, *, \Lambda_\Q, \psi)$.
For an open compact subgroup $K \subseteq G_0(\Zhat)$ let $M^B_K$ be the canonical model \cite{DeligneSV}*{Corollaire 2.7.21} of $G_0(\Q) \backslash X_0 \times G_0(\A_f) / K$ over the reflex field, which is $\Q$ in this case by \cite{MilneSh}*{Example 12.4.c}, which cites an unpublished paper of Deligne.
This is a normal projective variety by the theorem of Baily--Borel (see, for example, Theorem 3.12, Remark 3.13.(a), and Lemma 5.13 of \cite{MilneSh}) and by \cite{MilneSh}*{Theorem 3.3.(b)}.
Also let $\mathscr{M}^B_K$ be the fibered category over $\Q$ classifying the $*$-polarized QM-abelian varieties with level-$K$ structures.
This is a smooth proper Deligne--Mumford stack of dimension $d$ by \cite{Lan}*{Theorem 1.4.1.12, Theorem 2.2.4.13} and by \cite{BreenLabesse}*{III}, whose coarse moduli space is $M^B_K$ by \cite{Kottwitz}*{Section 8}.
Also let $\mathscr{M}^{B,\text{rat}}_K$ be the fibered category over the category of locally noetherian $\Q$-schemes which associates each scheme $S$ the groupoid whose objects are pairs $(A,i,\lambda,\alpha)$, where $A/S$ is a $2d$-dimensional abelian scheme, $i \colon B \to \End (A/S) \otimes \Q$, $\lambda$ is a $\Q^*$-polarization (see \cite{Lan}*{Definition 1.3.2.19}) preserving $i$, and $\alpha$ is a rational level-$K$ structure of $(A,i,\lambda)$ (\cite{Lan}*{Definition 1.3.8.7}), and whose morphisms are quasi-isogenies.
(For more details see \cite{Lan}*{Definition 1.4.2.1}.)
Then the pullback of $\mathscr{M}^B_K$ from the category of $\Q$-schemes to the category of locally noetherian $\Q$-schemes is canonically isomorphic to $\mathscr{M}^{B,\text{rat}}_K$ \cite{Lan}*{Proposition 1.4.3.4}.
In particular, identifying them, for an algebraically closed field (of characteristic zero) $k$, the objects of $\mathscr{M}^B_K(k)$ consist of the pairs $(A,i,\lambda, \alpha K)$, where $(A,i,\lambda)$ are as above, and $\alpha \colon \Lambda \otimes \A_f \to \hat{V} A$ are symplectic isomorphisms of symplectic $B \otimes \A_f$-modules.

For every open compact subgroup $K \subseteq G_0(\Zhat)$, there are no $\R$-rational points on $M^B_K$ by \cite{Shimura}*{Theorem 0} and by \cite{kj}*{Theorem 3.2}.
For simplicity, we just denote $M^B_K$ by $M^B$ for $K = G_0(\hat{\Z})$.

For more details, see \cite{Milne}, \cite{Lan}*{Chapter 1}, and \cite{kj}*{Section 2, 3}.

We define the Shimura covering, which is one of the main objects in this paper, following \cite{SkorobogatovShimuraCovering}*{Section 1} and \cite{dVP}.
In order to define it as an appropriate subcovering of a covering of $M^B$, we study the automorphism groups of our Shimura varieties.

Let $(G,X)$ be a Shimura datum, $E$ the reflex field, $K$ an open compact subgroup of $G(\A_f)$, $\operatorname{Sh}_K$ the canonical model of the Shimura variety $G(\Q) \backslash X \times G(\A_f) / K$ of level $K$ over $E$, and let $N(K)$ be the normalizer of $K$ in $G(\A_f)$.
Define the map $\rho_K \colon N(K) \to \Aut_E \operatorname{Sh}_K$ as $\beta \mapsto ([x,g] \mapsto [x, g\beta])$, which is defined over $E$ by \cite{MilneSh}*{Theorem 13.6}.
In \cite{dVP}, the author calls the images of $\rho_K$ the modular automorphisms.
When no confusion can arise, we just write $\rho_K$ by $\rho$.

\begin{lemma} \label{ker_rho=ZK}
Let $K$ be an open compact subgroup of $G_0(\A_f)$.
Then $\ker \rho_K = Z_0(\Q) K$.
\end{lemma}

\begin{proof}
Let $\beta \in N(K)$.
The map $\rho(\beta)$ sends the isogeny class of $(A,i,\lambda,\eta K) \in \mathscr{M}^{B,\text{rat}}_K(\C)$ to the isogeny class of $(A,i,\lambda,\eta \beta K)$ by the construction of $\mathscr{M}^{B,\text{rat}}_K \to M^B_K$ \cite{Kottwitz}*{Section 8}.
Hence $\rho(\beta)$ sends the isogeny class of $(A,i,\lambda,\eta K)$ to itself if and only if there exists $f \in \End_B^0 A$ and $k \in K$ such that $\lambda = r f^\vee \lambda f$ for some positive rational number $r$ and that the following diagram is commutative:
\begin{equation*}
\begin{aligned}
\xymatrix{
\Lambda \otimes_\Z \A_f \ar[r]^-{\eta} \ar[d]^{\beta k} & \hat{V} A \ar[d]^{f} \\
\Lambda \otimes_\Z \A_f \ar[r]^-{\eta} & \hat{V} A.
}
\end{aligned}
\end{equation*}
Now by \cite{kj}*{Proposition 4.2} if $A$ has no CM then $\End_B^0 A = F$.
Since there exists a non-CM $*$-polarized QM-abelian variety $A$ over $\C$, it shows that $\rho(\beta) = 1$ if and only if $\beta \in Z_0(\Q) K$.
\end{proof}

\begin{lemma} \label{Aut_Sh/Sh}
Let $L \subseteq K \subseteq G_0(\Zhat)$ be open compact subgroups and assume that $L$ is normal in $K$.
Then the map $\rho_L$ gives the isomorphism $K/(K \cap \pm L) \isom \Aut_\Q M^B_L/M^B_K$.
\end{lemma}

\begin{proof}
First by \cref{ker_rho=ZK}, the kernel of $\rho_L|_K$ is $Z_0(\Q) L \cap K = (Z_0(\Q) \cap K)L$.
Now the intersection of $Z_0(\Q)$ with $G_0(\hat{\Z})$ is $Z_0(\Q) \cap Z_0(\A_f) \cap G_0(\hat{\Z}) = Z_0(\Q) \cap Z_0(\hat{\Z})$, which is, since for each number field $E$ we have $E^* \cap \hat{O_E}^* = O_E^*$, equal to $Z_0(\Z) = \{ \pm 1 \}$.
Thus $\ker (\rho_L|_K) = (\{ \pm 1 \} \cap K)L$, and it is easy to show that it equals to $K \cap \pm L$.
Next we show that $\rho_L$ maps $K$ onto $\Aut_\Q M^B_L/M^B_K$.
Let $H := K/\ker \rho_L$.
Then the canonical map $M^B_L \to \operatorname{Sh}_K$ induces $M^B_L/H \to \operatorname{Sh}_K$.
At the $\C$-rational points, considering their double coset representations, this map is obviously an isomorphism.
Hence every automorphism of $M^B_L/M^B_K$ is in the image of $\rho_L$.
\end{proof}

Let $p_F$ be a prime of $F$ dividing $\mathfrak{d}_{B/F}$, $p$ the rational prime below $p_F$, $r$ the inertia degree of $p_F$ over the rationals, $p_B$ the unique maximal two-sided ideal (which is, in this case, equivalently the unique maximal one-sided ideal) of $O_{B, p_F}$, and let $\F_{p_B}$ be the residue field of $O_B$ at $p_B$, which is a quadratic extension of $\F_{p_F}$.
(For details see \cite{Voight}*{Chapter 13}.)

\begin{lemma} \label{im_=_N^-1}
The image of $G_0(\Zhat) \to \F_{p_B}^*$ is $N_{\F_{p_B}/\F_{p_F}}^{-1}(\F_p^*)$.
\end{lemma}

\begin{proof}
Since
\begin{equation*}
\begin{aligned}
\xymatrix{
O_{B,p_F}^* \ar[r] \ar[d]^{\Nrd} & \F_{p_B}^* \ar[d]^{N_{\F_{p_B}/\F_{p_F}}} \\
O_{F,p_F}^* \ar[r] & \F_{p_F}^*
}
\end{aligned}
\end{equation*}
is commutative, we have that the image of $G_0(\Zhat) \to \F_{p_B}^*$ is contained in $N_{\F_{p_B}/\F_{p_F}}^{-1}(\F_p^*)$.
Conversely let $w \in N_{\F_{p_B}/\F_{p_F}}^{-1}(\F_p^*)$.
We show that $w$ lifts to $G_0(\Zhat)$.
Let $b \in O_{B,p_F}^*$ be a lift of $w$.
We may assume that $\Nrd(b) \in \Z_p^*$.
For, there exists $a \in \Z_p^*$ such that $\Nrd(b)^{-1} a \in 1 + p_F$.
Take an unramified quadratic extension $L$ of $F_{p_F}$ inside $B_{p_F}$.
Then for example by \cite{Neukirch}*{Chapter V, Corollary 1.2}, there exists $\eta \in 1 + \mathfrak{p}$ such that $N_{L/F_{p_F}}(\eta) = \Nrd(b)^{-1}a$, where $\mathfrak{p}$ is the prime of $L$.
Thus replacing $b$ by $b\eta$ we may assume that $\Nrd(b) \in \Z_p^*$.
For each prime $\mathfrak{q}$ of $F$ dividing $p$ but not equal to $p_F$, by \cite{Voight}*{Lemma 13.4.9} there exists $\beta_\mathfrak{q} \in O_{B,\mathfrak{q}}^*$ which is mapped to $\Nrd(b)$ under $\Nrd_{B_\mathfrak{q}/F_\mathfrak{q}}$.
Let $\beta$ be the element of $(O_B \otimes \Z_p)^*$ whose component at $p_F$ is $b$ and at other $\mathfrak{q} \mid p$ is $\beta_\mathfrak{q}$.
Then $\beta \in G_0(\Z_p)$ and is mapped to $w$ under the canonical map $G_0(\Z_p) \to \F_{p_B}^*$.
It shows the statement.
\end{proof}

Let $K_{p_F}$ be the kernel of the canonical map $G_0(\Zhat) \to \F_{p_B}^*$.
Then from these lemmata, $\rho_{K_{p_F}}$ gives $N_{\F_{p_B}/\F_{p_F}}^{-1}(\F_p^*) / \pm 1 \isom \Aut_\Q M^B_{K_{p_F}}/M^B$.
Next we study the ramification of $M^B_{K_{p_F}}/M^B$.

We recall from \cite{kj} that $n_{\lcm} = \lcm \{ m : [F(\zeta_m) : F] \le 2 \}$, where $\zeta_m$ is a primitive $m$-th roof of $1$.

\begin{lemma} \label{H_divides_gcd}
Let $k$ be a field of characteristic zero, $A/k$ a QM-abelian variety, and $H$ be a finite subgroup of $\Aut_{O_B} A/k$.
Let $\mathfrak{q} \mid \mathfrak{d}_{B/F}$ be a prime of $F$, $\ell^f := \# \F_\mathfrak{q}$, and $\ell^a$ be the maximal $\ell$-power such that $\phi(\ell^a) \mid 2d$.
Then $H$ is a cyclic group of order dividing $n_{\lcm}$ and $(\ell^f + 1)\ell^a$.
\end{lemma}

\begin{proof}
First if $A$ has no CM, then by \cite{kj}*{Corollary 4.15} $H \subseteq \pm 1$, and hence the statement is trivial.
Next assume that $A$ has CM and the order of $H$ is larger than $2$.
By \cite{kj}*{Corollary 4.17} $H$ is cyclic and its order divides $n_{\lcm}$.
Hence it suffices to show that the order of $H$ divides $(\ell^f + 1)\ell^a$.
By \cite{kj}*{Proposition 4.16} $F' := \End_B^0 A$ is a totally imaginary quadratic extension of $F$ and splits $B/F$.
Hence $\mathfrak{q}$ does not split in $F'$.
Let $\mathfrak{Q}$ be the prime of $F'$ above $\mathfrak{q}$.
We bound the orders of the kernel and the image of $H \subseteq O_{F'}^* \to \F_\mathfrak{Q}^*$ from above.
The following diagram is commutative:
\begin{equation*}
\begin{aligned}
\xymatrix{
O_{F'}^* \ar[r] \ar[d]^{N_{F'/F}} & O_{F',\mathfrak{Q}}^* \ar[d]^{N_{F'_\mathfrak{Q}/F_\mathfrak{q}}} \\
O_F^* \ar[r] & O_{F,\mathfrak{q}}^*.
}
\end{aligned}
\end{equation*}
Hence so is
\begin{equation*}
\begin{aligned}
\xymatrix{
O_{F'}^* \ar[r] \ar[d]^{N_{F'/F}} & \F_\mathfrak{Q}^* \ar[d] \\
O_F^* \ar[r] & \F_\mathfrak{q}^*,
}
\end{aligned}
\end{equation*}
where the right vertical map is the square map or the norm if the inertia degree of $\mathfrak{Q}/\mathfrak{q}$ is one or two respectively.
Now if we let $\zeta$ be a generator of $H$ then $F' = F(\zeta)$ by the assumption that $\# H > 2$.
Since $\zeta + \zeta^{-1}$ is real and since $F$ is the maximal totally real subfield of $F'$, we have that $\zeta + \zeta^{-1} \in F$, and thus $N_{F'/F}(\zeta) = 1$.
Hence under the map $O_{F'}^* \to \F_\mathfrak{Q}^*$ the group $H$ is mapped to the kernel of $\F_\mathfrak{Q}^* \to \F_\mathfrak{q}^*$.
On the other hand, since the order of the kernel of $O_{F'}^* \to \F_\mathfrak{Q}^*$ is a power of $\ell$, the order of the kernel of $H \to \F_\mathfrak{Q}^*$ divides $\ell^a$.
Therefore the order of $H$ divides $\ell^a \# \ker(\F_\mathfrak{Q}^* \to \F_\mathfrak{q}^*)$, which divides $\ell^a (\ell^f + 1)$.
\end{proof}

\begin{lemma} \label{ram=Aut/2}
Let $K \subseteq G_0(\Zhat)$ be an open compact subgroup, $k$ an algebraically closed field of characteristic zero, and $X \in \mathscr{M}^B_K(k)$.
Then the ramification index of the canonical map $\mathscr{M}^B_K \to M^B_K$ at $X$ is $\# \Aut X$ if $-1 \not\in K$ and is $\# \Aut X / 2$ if $-1 \in K$.
\end{lemma}

\begin{proof}
Take a sufficiently small open compact subgroup $L$ of $K$, in particular $-1 \not\in L$, and fix a positive integer $N$ such that $L$ contains the kernel of $G_0(\Zhat) \to G_0(\Z/N)$.
Let $Y$ be an object of $\mathscr{M}^B_L(k)$ mapped to $X$, $Q \in M^B_L(k)$ the image of $Y$, and $P \in M^B_K(k)$ be the image of $X$.
Write $Y = (A,i,\lambda, \alpha L)$, where 
\begin{equation*}
\alpha \in \operatorname{Im} (\operatorname{SymIsom}_{O_B \otimes \Zhat}(\Lambda \otimes \Zhat, \hat{T}A) \to \operatorname{SymIsom}_{O_B \otimes \Z/N}(\Lambda/N, A[N](k)) ).
\end{equation*}
Then since the map $\mathscr{M}^B_L \to \mathscr{M}^B_K$ is etale and the map $\mathscr{M}^B_L \to M^B_L$ is an isomorphism, the ramification index which we want to calculate is the one of $\pi \colon M^B_L \to M^B_K$ at $P$.
We have seen that $\Aut(M^B_L/M^B_K)$ is isomorphic to $K/\pm L \cap K$, which is isomorphic to $\pm K / \pm L$ in the both case that $-1 \in K$ or not, by $\rho_L$ (\cref{Aut_Sh/Sh}).
Hence the inertia group of $Q \mapsto P$ is isomorphic to $\{ g \in  \pm K / \pm L | \rho_L(g)(Q) = Q\}$ under $\rho_L$.

Define the action of $\Aut X$ on $K/L$ from left by $u \cdot gL := \alpha^{-1} u \alpha gL$.
It is easy to show that this action is free: for $u \in \Aut X$ and for $g \in K$, the element $u \cdot gL$ is equal to $gL$ itself if and only if $u$ is an automorphism of $(A,i,\lambda,\alpha g L) \in \mathscr{M}^B_L(k)$.
Now since we choose $L$ so that $\mathscr{M}^B_L$ is fine, it exactly means that $u = 1$.
Also it is straightforward to see that $\Aut X \backslash K/L \isom \pi^{-1}(P)$ by sending the class of $g \in K$ to $\rho_L(g)(Q)$, which is equal to the isomorphism class of $(A,i,\lambda,\alpha gL)$.
Therefore the inertia group of $Q \mapsto P$ is isomorphic to the preimage of $1$ under the map $\pm K/\pm L \to \Aut X \backslash K/L$, and hence the inertia degree is equal to $\# \Aut X$ if $-1 \not\in K$ or $\# \Aut X / 2$ if not, since $-1 \not\in L$ and since the action of $\Aut X$ on $K/L$ is free.
\end{proof}

\begin{proposition} \label{M^B_p_F_is_a_scheme}
Let $\mathfrak{n}_F$ be the product of $\mathfrak{D}_{F/\Q}$ and $N_{F(\zeta)/F}(\zeta-1)$ for all nontrivial roots $\zeta$ of unity in $\Qbar$ such that $[F(\zeta) : F] \le 2$.
Then $\mathfrak{n}_F$ is a well-defined integral ideal and if $p_F \nmid \mathfrak{n}_F$ then the algebraic stack $\mathscr{M}^B_{K_{p_F}}$ is a scheme.
\end{proposition}

\begin{proof}
The well-definedness is trivial.
Assuming that $p_F \nmid \mathfrak{n}_F$, we show that $\mathscr{M}^B_{K_{p_F}}$ is represented by a scheme.
For every open compact subgroup $K \subseteq G_0(\hat{\Z})$, the coarse moduli space of $\mathscr{M}^B_K$ is a scheme, and is $M^B_K$.
(See \cite{Kottwitz}*{Section 8}.)
Thus in order to show the statement it suffices to show that the algebraic stack $\mathscr{M}^B_{K_{p_F}}$
is an algebraic space.
Let $k$ be an algebraically closed field of characteristic zero, $X$ an object of $\mathscr{M}^B_{K_{p_F}}(k)$,
and $(A,i,\lambda)$ the underlying $*$-polarized QM-abelian variety of $X$.
It suffices to show that the automorphism group $\Aut X$ is trivial.
Since the set of level-$K_{p_F}$ structures on $(A,i,\lambda)$ is
\begin{equation*}
\operatorname{Im} (\operatorname{SymIsom}_{O_B \otimes \Zhat}(\Lambda \otimes \Zhat, \hat{T}A) \to \operatorname{SymIsom}_{O_B \otimes \Z/p}(\Lambda/p, A[p](k)) ) / K_{p_F},
\end{equation*}
which, fixing a generator of the $O_B$-module $\Lambda$, injects into $A[p_B](k)$, the automorphism group $\Aut X$ is the kernel of the canonical map $\Aut (A,i,\lambda) \to \Aut A[p_B](k)$.
We show that this map is injective.
Now $O_F$ is contained in $\End_{O_B} A$, which is an order $O$ of $F$ or of a totally imaginary quadratic extension $F'$ by \cite{kj}*{Proposition 4.12, Proposition 4.16}, $\End_{O_B} A[p_B] = \F_{p_B}$ by \cite{kj}*{Lemma 5.2} since $p_F \nmid \mathfrak{D}_{F/\Q}$, and the kernel of the canonical map $\End_{O_B} A \to \End_{O_B} A[p_B]$ is a maximal ideal $\mathfrak{p}$ of $O$ containing $p_F$.
As a consequence, an element $\zeta \in \Aut(A,i,\lambda)$, which satisfies $[F(\zeta):F] \le 2$ by \cite{kj}*{Corollary 4.17}, is mapped to $1$ under $\Aut (A,i,\lambda) \to \Aut A[p_B](k)$ if and only if $\zeta - 1 \in \mathfrak{p}$, which implies that $p_F \mid N_{F(\zeta)/F}(\zeta - 1)$.
It completes the proof.
\end{proof}

\begin{proposition} \label{ramification_index_of_the_map}
Let $k$ be an algebraically closed field of characteristic zero.
Assume that $p_F \nmid n_{\lcm} \mathfrak{n}_F$.
Then for every $k$-rational point, the ramification index of the canonical map $M^B_{K_{p_F}} \to M^B$ divides $\gcd(n_{\lcm}, \#\F_{p_F} + 1)/2$.
\end{proposition}

\begin{proof}
By \cref{H_divides_gcd,ram=Aut/2,M^B_p_F_is_a_scheme}.
\end{proof}

Now we are ready to define the Shimura covering of $M^B$.
Let $H$ be the unique subgroup of
\begin{equation*}
N_{\F_{p_B}/\F_{p_F}}^{-1}(\F_p^*) / \pm 1 \isom \Aut_\Q (M^B_{K_{p_F}} / M^B)
\end{equation*}
(see \cref{Aut_Sh/Sh,im_=_N^-1})
of order $\gcd(n_{\lcm}, (p-1)(\#\F_{p_F} + 1))/2$,
and let $X_{p_F}$ be the quotient of $M^B_{K_{p_F}}$ by $H$.

\begin{proposition} \label{Shimura_covering}
Assume that $p_F \nmid n_{\lcm} \mathfrak{n}_F$.
Then the covering $X_{p_F} \to M^B$ as above is finite etale, and moreover is a torsor under the constant group $N_{\F_{p_B}/\F_{p_F}}^{-1}(\F_p^*)^{n_{\lcm}}$ in etale topology.
\end{proposition}

\begin{proof}
By the definition of $H$ and by \cref{Aut_Sh/Sh,im_=_N^-1}, we have that
\begin{equation*}
\begin{aligned}
\Aut_\Q (M^B_{K_{p_F}} / M^B)/H & \isom N_{\F_{p_B}/\F_{p_F}}^{-1}(\F_p^*) / H \\
& \isom N_{\F_{p_B}/\F_{p_F}}^{-1}(\F_p^*)^{n_{\lcm}},
\end{aligned}
\end{equation*}
where the second isomorphism is the $n_{\lcm}$-th power map.
Thus by \cref{ramification_index_of_the_map}, the group $\Aut_\Q (M^B_{K_{p_F}} / M^B)/H$ acts freely on $M^B_{K_{p_F}}(\C)$.
Hence for example by \cite{MumfordAV}*{Section 7, Theorem} the map $X_{p_F} \to M^B$ is etale.
\end{proof}

We call $X_{p_F}$ the Shimura covering of $M^B$ in the case that $p_F \nmid n_{\lcm} \mathfrak{n}_F$.

We can study geometric properties of $M^B_{K_{p_F}}$ following \cite{SkorobogatovShimuraCovering}*{Section 1} and \cite{dVP}, for example:
the set of connected components of $(M^B_{K_{p_F}})_\C$ is canonically isomorphic to $\F_p^*$ (and hence $X_{p_F}$ is geometrically connected)
and for a connected component $U$ of $(M^B_{K_{p_F}})_\C$, we have that $\Aut_\C (U/M^B_\C)$ is canonically isomorphic to $\ker N_{\F_{p_B}/\F_{p_F}}/\pm 1$.
These seem to be interesting in their own rights.
None of them, however, will be used in this paper, hence we do not discuss it.

\section{The canonical characters of QM-abelian varieties} \label{Section:canonical_characters}

Unless otherwise stated, keep the notations as in the previous section.
In particular $p_F$ is a prime of $F$ dividing $\mathfrak{d}_{B/F}$, $p$ is the rational prime below $p_F$, $r$ is the inertia degree of $p_F$ over the rationals, $p_B$ is the unique maximal two-sided ideal of $O_{B, p_F}$, and $\F_{p_B}$ is the residue field of $O_B$ at $p_B$.
Also $\mathfrak{n}_F$ is the ideal of $F$ described in \cref{M^B_p_F_is_a_scheme}.
In this section we assume that $p_F \nmid n_{\lcm} \mathfrak{n}_F$.

In order to study the rational points of $M^B$, we define some characters obtained from them.
Let $k$ be a field of characteristic zero and $X \in \mathscr{M}^B(k)$ an object.
Let us define a character $\rho_X \colon G_k \to N_{\F_{p_B}/\F_{p_F}}^{-1}(\F_p^*)$ as follows:
The absolute Galois group $G_k$ acts on the set of the level $K_{p_F}$-structures of $X_\kbar$, i.e., on
\begin{equation*}
\operatorname{Im} ( \operatorname{SymIsom}_{O_B \otimes \Zhat}(\Lambda \otimes \hat{\Z}, \hat{T} A) \to \operatorname{SymIsom}_{O_B \otimes \Z/p}(\Lambda/p, A[p]) ) /K_{p_F}.
\end{equation*}
Also $G_0(\Zhat)/K_{p_F} \isom N_{\F_{p_B}/\F_{p_F}}^{-1}(\F_p^*)$ (see \cref{im_=_N^-1}) acts transitively and freely.
Since these two actions are compatible with each other, we obtain
\begin{equation*}
G_k \to \Aut_{G_0(\Zhat)} ( \text{the set of level $K_{p_F}$-structures on $X_\kbar$} ),
\end{equation*}
whose target is canonically isomorphic to $G_0(\Zhat)/K_{p_F} \isom N_{\F_{p_B}/\F_{p_F}}^{-1}(\F_p^*)$.
In \cite{Jordan}, the character $\rho_X$ is called the canonical character in the case of $F = \Q$.

\begin{remark} \label{rho_X_and_rho_A}
Under the situation above, let $A$ be the underlying QM-abelian variety of $X$.
Since now we are assuming that $p_F$ is unramified, $A[p_B](\kbar)$ is free of rank one over $\F_{p_B}$ by \cite{kj}*{Lemma 5.2}, and hence its Galois action gives $G_k \to \F_{p_B}^*$, which we denote by $\rho_A$.
After compositing with the inclusion $N_{\F_{p_B}/\F_{p_F}}^{-1}(\F_p^*) \subseteq \F_{p_B}^*$, it is easy to show that $\rho_X = \rho_A$.
\end{remark}

On the other hand let $x \in M^B(k)$.
If $x$ has a model $X$ over $k$, of course we obtain the character $\rho_X$ of $G_k$ as above.
However, even if the point $x$ has no models over $k$ itself, we can define a character of $G_k$, which is something like $\rho_X^{n_{\lcm}}$, as follows:
Pulling back the covering $X_{p_F} \to M^B$ at $x$, we obtain the twist of the constant group $N_{\F_{p_B}/\F_{p_F}}^{-1}(\F_p^*)^{n_{\lcm}}$ over $k$ by \cref{Shimura_covering}, which corresponds to an element of the cohomology group $H^1(k, N_{\F_{p_B}/\F_{p_F}}^{-1}(\F_p^*)^{n_{\lcm}})$, i.e., the character $G_k \to N_{\F_{p_B}/\F_{p_F}}^{-1}(\F_p^*)^{n_{\lcm}}$, which we denote by $\varphi_x$.
If $x$ has a model over $k$, then the following, which is well-known in the case of $F = \Q$ (\cite{SkorobogatovShimuraCovering}*{Lemma 2.1}), holds:

\begin{lemma} \label{rho^n=varphi}
Let $k$ be a field of characteristic zero and let $x \in M^B(k)$.
If $x$ has a model $X$ over $k$, then we have $\rho_X^{n_{\lcm}} = \varphi_x$.
\end{lemma}

\begin{proof}
By moduli interpretation and by the definition of $X_{p_F}$.
\end{proof}

Using the lemma above and using a moduli interpretation, it is easy to study the behavior of these characters.

\begin{lemma} \label{varphi_is_nr}
Let $\ell$ be a rational prime different from $p$, $k$ a finite extension of $\Q_\ell$, and let $x \in M^B(k)$.
Then $\varphi_x$ is unramified. 
\end{lemma}

\begin{proof}
At first $x$ has a model over $k^\nr$ since $k^\nr$ has cohomological dimension one (e.g., see \cite{Fu}*{Theorem 4.5.9}).
(See also the proof of \cite{kj}*{Lemma 6.5}.)
Hence there exists $X \in \mathscr{M}^B(k^\nr)$ such that $\varphi_x |_{I_k} = \rho_X^{n_{\lcm}}$ by \cref{rho^n=varphi}.
The rest of the proof is completely the same as the one of \cite{kj}*{Lemma 5.9}.
For completeness, we give a proof. 

Let $A/k^\nr$ be the underlying QM-abelian variety of $X$.
If $A/k^\nr$ has good reduction, then since the inertia group $I_k$ acts on $A[p](\kbar)$ trivially, the proposition follows.
Assume that $A/k$ has bad reduction.
By \cite{kj}*{Proposition 5.6, 5.8}, there exists a totally ramified Galois extension $M$ of $k^\nr$ whose exponent divides $n_{\lcm}$ and such that
$A_M$ has good reduction over $M$.
By the argument above we have that $\rho_X |_{G_M} = 1$,
and hence $\varphi_x|_{I_k} = 1$.
\end{proof}

On the other hand at a place dividing $p$, the character $\varphi_x$ is ramified in general.
However it is elemental to describe the behavior of its restriction to the inertia group.

For a local field $k$ and for a group homomorphism $\eta \colon G_k \to H$ to an abelian group $H$, we also denote the composite of $\eta$ and the local reciprocity map $k^* \to G_k^\ab$ by the same symbol $\eta$.
Also we denote the composite of $\eta |_{I_k}$ with $O_k^* \to I(k^\ab/k)$, which is $\eta |_{O_k^*}$, by the same symbol $\eta |_{I_k}$.

Before stating the results, recall that $\# \F_{p_F} = p^r$.

\begin{lemma} \label{rho_at_p}
Let $k/\Q_p$ be a finite extension of inertia degree $f$, $\mathfrak{p}$ the prime of $k$, $X \in \mathscr{M}^B(k)$, and let $t := \gcd(2r, f)$.
Fix $\F_{p_B} \to \overline{\F_\mathfrak{p}}$.
Then there exists an integer $c$ such that $c(p-1)(p^r+1)$ is divisible by $p^t - 1$ and that for every $u \in O_k^*$ we have
$\rho_X(u) = N_{\F_\mathfrak{p}/\F_{p^t}}(u \mod \mathfrak{p})^{-c}$, where $\F_{p^t}$ is the subfield of $\F_{p_B}$ of order $p^t$. 
\end{lemma}

\begin{proof}
Since the kernel of $O_k^* \to \F_\mathfrak{p}^*$ is pro-$p$ and since the order of the target of $\rho_X$ is prime-to-$p$, the character $\rho_X|_{I_k}$ induces $\F_\mathfrak{p}^* \to N_{\F_{p_B}/\F_{p_F}}^{-1}(\F_p^*) \subseteq \F_{p_B}^*$.
This character factors through, since $\gcd(p^f - 1, p^{2r} - 1) = p^t - 1$, the group $\F_{p^t}^*$.
Hence it factors through $N_{\F_\mathfrak{p}/\F_{p^t}}^{-c} \colon \F_\mathfrak{p}^* \to \F_{p^t}^*$ for an integer $c$.
Finally since of course the image of $N_{\F_{p_B}/\F_{p_F}} \circ \rho_X$ is contained in $\F_p^*$, we have that $N_{\F_\mathfrak{p}/\F_{p^t}}^{-c(p-1)(p^r+1)} = 1$.
\end{proof}

For a field $k$ on which $p$ is invertible, let $\chi$ be the mod-$p$ cyclotomic character.

\begin{lemma} \label{rho_at_p_more_about_c}
Keep the notations as in \cref{rho_at_p} and let $e$ be the ramification index of $k$ and $c$ be an integer as in \cref{rho_at_p}.
Assume that $N_{\F_{p_B}/\F_{p_F}} \circ \rho_X^{n_{\lcm}} |_{I_k} = \chi^{n_{\lcm}}$.
Then $ern_{\lcm} \equiv \frac{2r}{t}cn_{\lcm} \mod p-1$.
\end{lemma}

\begin{proof}
Let $u \in O_k^*$.
Then by the assumption we have
\begin{equation*}
\begin{aligned}
N_{\F_{p_B}/\F_{p_F}} \circ \rho_X (u)^{n_{\lcm}} & = \chi(u)^{n_{\lcm}} \\
& = N_{\F_\mathfrak{p}/\F_p}(u \mod \mathfrak{p})^{-en_{\lcm}}.
\end{aligned}
\end{equation*}
On the other hand $\rho_X (u) = N_{\F_\mathfrak{p}/\F_{p^t}}(u \mod \mathfrak{p})^{-c}$ by \cref{rho_at_p}, and hence
\begin{equation*}
N_{\F_{p_B}/\F_{p_F}} \circ \rho_X (u) = N_{\F_\mathfrak{p}/\F_p}(u \mod \mathfrak{p})^{-c(p^r + 1)(p-1)/(p^t - 1)}.
\end{equation*}
Combining them together we have $en_{\lcm} \equiv cn_{\lcm}(p^r + 1)(p-1)/(p^t - 1) \mod p-1$.
From this equation we get
\begin{equation*}
\frac{p^r - 1}{p-1}en_{\lcm} \equiv \frac{p^{2r}-1}{p^t-1}cn_{\lcm} \mod p - 1.
\end{equation*}
Now since $p^{ab}-1/p^a-1 \equiv b \mod p-1$ for any positive integers $a,b$, we obtain $ern_{\lcm} \equiv \frac{2r}{t}cn_{\lcm} \mod p-1$.
\end{proof}

\begin{lemma} \label{rho^blahblah=N^-rn}
Let $k/\Q_p$ be a finite extension of inertia degree $f$ and $X \in \mathscr{M}^B(k)$.
Assume that $N_{\F_{p_B}/\F_{p_F}} \circ \rho_X^{n_{\lcm}} |_{I_k} = \chi^{n_{\lcm}}$ and that $\gcd(2r,f) = r$.
Then 
\begin{equation*}
\rho_X (u)^{2n_{\lcm}r} = N_{k/\Q_p}(u)^{-n_{\lcm}r} \mod p
\end{equation*}
for any $u \in O_k^*$.
\end{lemma}

\begin{proof}
By the assumption the integer $t$ as in \cref{rho_at_p} is $r$.
Let $\mathfrak{p}$ be the prime of $k$ and fix $\F_{p_B} \to \overline{\F_\mathfrak{p}}$.
Let $c$ be an integer as in \cref{rho_at_p} and let $c'$ be an integer such that $2c = \frac{p^r-1}{p-1}c'$.
Let $u \in O_k^*$.
Then by \cref{rho_at_p} we have that
\begin{equation*}
\begin{aligned}
\rho_X (u)^{2n_{\lcm}r} & = N_{\F_\mathfrak{p}/\F_{p^r}}(u \mod \mathfrak{p})^{-2n_{\lcm}rc} \\
& = N_{\F_\mathfrak{p}/\F_p}(u \mod \mathfrak{p})^{-n_{\lcm}rc'}.
\end{aligned}
\end{equation*}
By \cref{rho_at_p_more_about_c} we have $ern_{\lcm} \equiv 2cn_{\lcm} \mod p-1$.
Since by the definition $2c \equiv rc' \mod p-1$, we obtain that $n_{\lcm} r c' \equiv en_{\lcm}r \mod p-1$.
Therefore
\begin{equation*}
\begin{aligned}
N_{\F_\mathfrak{p}/\F_p}(u \mod \mathfrak{p})^{-n_{\lcm}rc'} & = N_{\F_\mathfrak{p}/\F_p}(u \mod \mathfrak{p})^{-en_{\lcm}r} \\
& = N_{k/\Q_p}(u)^{-rn_{\lcm}} \mod p.
\end{aligned}
\end{equation*}
\end{proof}

\begin{lemma} \label{Nrho=chi}
Let $k$ be a number field, $m$ an integer, and $\eta \colon G_k \to N_{\F_{p_B}/\F_{p_F}}^{-1}(\F_p^*)$ a character satisfying that for every place $v$ of $k$ there exists $X_v \in \mathscr{M}^B(k_v)$ such that $\eta|_{G_{k_v}} = \rho_{X_v}^m$.
Then $N_{\F_{p_B}/\F_{p_F}} \circ \eta = \chi^m$.
\end{lemma}

\begin{proof}
Let $v$ be a place of $k$ not dividing $p$.
It suffices to show that $N_{\F_{p_B}/\F_{p_F}} \circ \eta (\Frob_v) = (\# \F_v)^m \mod p$.

Let $X_v \in \mathscr{M}^B(k_v)$ be such that $\eta |_{G_{k_v}} = \rho_{X_v}^m$ and let $A_v$ be the underlying QM-abelian variety over $k_v$ of $X_v$.
There exists a totally ramified finite extension $L_v$ of $k_v$ on which $A_v$ has good reduction and such that $\Frob_v$ is also a Frobenius of $L_v$ by \cite{kj}*{Proposition 5.6} (and by the proof of \cite{Jordan}*{Proposition 3.2}).
Letting $A'_v$ be the Neron model of $A_{v,L_v}$, the element $\eta(\Frob_v)$, which is equal to $\rho_{A_v}(\Frob_v)^m$ by the assumption and by \cref{rho_X_and_rho_A}, corresponds to $\rho_{A'_{v,\F_v}}(\Frob_{\F_v})^m = T_p(\pi)^m \mod p_B$ under the canonical isomorphism $A_v[p_B](\overline{k_v}) \isom A'_v[p_B](\overline{\F_v})$, where $\pi$ is the relative Frobenius endomorphism of $A'_{v,\F_v}/\F_v$.
Since
\begin{equation*}
\begin{aligned}
\xymatrix{
O_{B,p_F}^* \ar[r] \ar[d]^{\Nrd} & \F_{p_B}^* \ar[d]^{N_{\F_{p_B}/\F_{p_F}}} \\
O_{F,p_F}^* \ar[r] & \F_{p_F}^*
}
\end{aligned}
\end{equation*}
is commutative, by \cite{kj}*{Proposition 4.20} and by its proof we have that $N_{\F_{p_B}/\F_{p_F}} \circ \eta (\Frob_v) = (\# \F_v)^m \mod p$, which is what we wanted.
\end{proof}

Piecing them together, we obtain global information of the character $\varphi_x$ for a rational point $x$ of $M^B$ over a number field.

\begin{proposition} \label{varphi(Frob)^blahblah=q^blahblah}
Let $k$ be a number field containing the normal closure of $F$ over $\Q$.
Let $\varphi \colon G_k \to N_{\F_{p_B}/\F_{p_F}}^{-1}(\F_p^*)^{n_{\lcm}}$ be a character such that for all place $v$ of $k$ there exists $x_v \in M^B(k_v)$ satisfying $\varphi|_{G_{k_v}} = \varphi_{x_v}$.
Let $u$ be $1$ or $2$ according to $B \otimes_\Q k \isom M_2(F \otimes k)$ or not respectively,
$\mathfrak{q}$ a prime of $k$ above a rational prime $q \ne p$, and $e$ and $f$ be its ramification index and inertia degree over $\Q$ respectively.
\begin{enumerate}
\item If for every prime $\mathfrak{p}$ of $k$ dividing $p$ the greatest common divisor of $2r$ and the absolute inertia degree of $\mathfrak{p}$ is $r$, then
\begin{equation} \label{varphi(Frob)^blahblah=q^blahblah:eq1}
\varphi(\Frob_\mathfrak{q})^{2uh'_kr} = q^{un_{\lcm}h'_krf} \mod p.
\end{equation}
\item If $\mathfrak{q}$ is nonsplit over $\Q$, i.e., if $ef = [k:\Q]$, then
\begin{equation} \label{varphi(Frob)^blahblah=q^blahblah:eq2}
\varphi(\Frob_\mathfrak{q})^{2uer} = q^{un_{\lcm}r[k:\Q]} \mod p.
\end{equation}
\end{enumerate}
\end{proposition}

Note that in this statement, since $M^B$ has no real points by \cite{kj}*{Theorem 3.2}, by the assumption the base field $k$ must be totally imaginary.

\begin{proof}
Let $J_k^{pO_k}$ be the ideal group relatively prime to $p$ of $k$.
Then by \cref{varphi_is_nr} the character $\varphi$ induces a character $J_k^{pO_k} \to N_{\F_{p_B}/\F_{p_F}}^{-1}(\F_p^*)^{n_{\lcm}}$, which we also denote by $\varphi$.
Define $K$ to be $k$ if $B \otimes_\Q k \isom M_2(F \otimes k)$ and to be a quadratic extension of $k$ in which the prime $\mathfrak{q}$ and the primes of $k$ dividing $\Delta'$ are ramified if $B \otimes_\Q k \not\isom M_2(F \otimes k)$ respectively.
Then $[K:k] = u$ and $B \otimes K \isom M_2(F \otimes K)$.
For, since we are assuming that $k$ contains the normal closure of $F/\Q$, we have that $F \otimes_\Q k \isom \prod_\tau k$, where $\tau$ runs over the embeddings of $F$ into $k$, and hence $B \otimes_\Q k \isom \prod_\tau B \otimes_{F,\tau} k$, whose each direct factor has the discriminant dividing $\tau(\mathfrak{d}_{F/\Q})O_k$.
By \cite{kj}*{Theorem 4.24}, for every place $v$ of $k$ and for every place $w$ of $K$ above $v$, there exists $X_w \in \mathscr{M}^B(K_w)$ representing a rational point $x_v \times K_w \in M^B(K_w)$.
By \cref{rho^n=varphi} for every place $w$ of $K$ we have that $\varphi|_{G_{K_w}} = \rho_{X_w}^{n_{\lcm}}$.

First we show \cref{varphi(Frob)^blahblah=q^blahblah:eq1}.
Since the inertia degree of $\mathfrak{q}$ in $K$ is one and since its ramification index is $u$, we have $\varphi(\mathfrak{q}^u) = \varphi|_{G_K}(\mathfrak{q}O_K)$.
Therefore if $\mathfrak{q}^{h'_k}$ is generated by $\alpha \in O_k$, then we obtain that $\varphi(\mathfrak{q}^{uh'_k}) = \varphi|_{G_K}(\alpha O_K)$, which is, since $K$ is totally imaginary, equal to $\prod_{\mathfrak{P}} \rho_{X_\mathfrak{P}}|_{I_\mathfrak{P}}(\alpha^{-1})^{n_{\lcm}}$, where $\mathfrak{P}$ runs over the primes of $K$ dividing $p$.
Fix such a prime $\mathfrak{P}$.
By \cref{Nrho=chi}, we have that $N_{\F_{p_B}/\F_{p_F}} \circ \varphi = \chi^{n_{\lcm}}$, and in particular $N_{\F_{p_B}/\F_{p_F}} \circ \rho_{X_\mathfrak{P}}^{n_{\lcm}}|_{I_\mathfrak{P}} = \chi^{n_{\lcm}}$.
Since the greatest common divisor of $2r$ and the absolute inertia degree of $\mathfrak{P}$ is $r$ by the assumption, it follows from \cref{rho^blahblah=N^-rn} that
\begin{equation*}
\rho_{X_\mathfrak{P}} |_{I_\mathfrak{P}} (\alpha^{-1})^{2n_{\lcm}r} = N_{K_\mathfrak{P}/\Q_p}(\alpha)^{n_{\lcm}r} \mod p.
\end{equation*}
Therefore
\begin{equation*}
\begin{aligned}
\varphi(\mathfrak{q})^{2uh'_kr} & = N_{K/\Q}(\alpha)^{rn_{\lcm}} \mod p \\
& = N_{k/\Q}(\alpha)^{urn_{\lcm}} \mod p.
\end{aligned}
\end{equation*}
Finally since $\mathfrak{q}^{h'_k} = \alpha O_k$, the absolute value of $N_{k/\Q}(\alpha)$ is $q^{h'_kf}$.
As a consequence, since $n_{\lcm}$ is even, we obtain the first equation of the statement.

Next supposing that $\mathfrak{q}$ is nonsplit we show \cref{varphi(Frob)^blahblah=q^blahblah:eq2}.
By the assumption we have that $\mathfrak{q}^e = qO_k$.
Thus 
\begin{equation*}
\begin{aligned}
\varphi(\mathfrak{q}^{ue}) & = \varphi|_{G_K}(qO_K) \\
& = \prod_{\mathfrak{P}} \rho_{X_\mathfrak{P}}|_{I_\mathfrak{P}}(q^{-1})^{n_{\lcm}},
\end{aligned}
\end{equation*}
where $\mathfrak{P}$ runs over the primes of $K$ dividing $p$, as in the proof of the first equation.
Fix such a prime $\mathfrak{P}$ and let $f$ and $e$ be its inertia degree and ramification index respectively, $t$ be the greatest common divisor of $2r$ and $f$,
fix $\F_{p_B} \to \overline{\F_\mathfrak{P}}$,
and let $c$ be an integer as in \cref{rho_at_p}.
Then by \cref{rho_at_p} we have
\begin{equation*}
\begin{aligned}
\rho_{X_\mathfrak{P}}(q^{-1})^{n_{\lcm}} & = N_{\F_\mathfrak{P}/\F_{p^t}}(q)^{c n_{\lcm}} \\
& = q^{\frac{c n_{\lcm} f}{t}} \mod p,
\end{aligned}
\end{equation*}
and hence it follows that $\rho_{X_\mathfrak{P}}(q^{-1})^{2n_{\lcm}r} = q^{n_{\lcm}r[K_\mathfrak{P}:\Q_p]} \mod p$ from \cref{rho_at_p_more_about_c}.
Therefore we obtain the second equation of the statement.
\end{proof}

Before studying the rational points of $M^B$ from global behavior of the characters, we need to prepare some exceptional sets.

\begin{definition}
Let $f$ and $e$ be positive integers and $\ell$ be a rational prime.
Let $FR(\ell^f)$ be the set
\begin{equation*}
\begin{aligned}
\left\{ \beta \in \Qbar \bigg | 
\begin{array}{c}
\beta^2 + b \beta + \ell^f = 0 \text{ for some } b \in O_F \text{ with } \\
| b |_v \le 2 \sqrt{\ell}^f \text{ for every infinite place } v \text{ of } F
\end{array}
\right\},
\end{aligned}
\end{equation*}
which contains the set of Weil numbers of QM-abelian varieties over $\F_{\ell^f}$ by \cite{kj}*{Proposition 4.20}, and let
\begin{equation*}
C(\ell^f, e) := \{ \beta^e + \overline{\beta}^e | \beta \in FR(\ell^f) \} \subseteq O_F,
\end{equation*}
where $\overline{\beta}$ is $\beta$ itself or the conjugate of $\beta$ over $F$ if $\beta \in F$ or not respectively.
Note that since $O_F$ is discrete in $F \otimes \R$ these are finite sets.
\end{definition}

As is expected, the trace of a Frobenius of a QM-abelian variety over a local field belong to $C(\ell^f,e)$.
Before we show it, we need some notations.
Let $k$ be a field, $A$ a $k$-abelian variety of dimension $2d$ on which an order of $B$ acts, and let $\ell \ne \ch k$ be a prime.
Fixing an isomorphism $\alpha$ between $V_\ell A$ and $B_\ell$ as $B_\ell$-modules (\cite{kj}*{Proposition 4.3}), we obtain an isomorphism $\theta \colon \End_B V_\ell A \to B_\ell^\opp$.
Using this isomorphism, for $f \in \End_B V_\ell A$,
we define the reduced trace and reduce norm of $f$ by
$\Trd (\theta(f))$ and by $\Nrd (\theta(f))$ respectively.
These definitions are independent of the choice of $\alpha$, thus we simply denote them by $\Trd f$ and by $\Nrd f$ respectively.
Note that if the order acting on $A$ is $O_B$ and if $f \in \End_{O_B} T_\ell A$, then $\Trd f$ and $\Nrd f$ are in $ O_{F,\ell}$.

\begin{lemma} \label{a_of_A/k}
Let $\ell$ be a rational prime, $k$ a finite extension of $\Q_\ell$ of inertia degree $f$, $A$ a QM-abelian variety over $k$, and let $m$ be a positive integer.
For a rational prime $\ell' \ne \ell$, let $R_{\ell'} \colon G_k \to \Aut_{O_B} T_{\ell'} A$ be the $\ell'$-adic representation.
Then $\Trd_{B_{\ell'}/F_{\ell'}} (R_{\ell'} (\Frob_k^m))$ lies in $C(\ell^f,m)$, is independent on the choice of $\ell' \ne \ell$, and $\Nrd_{B_{\ell'}/F_{\ell'}}(R_{\ell'}(\Frob_k^m)) = \ell^{mf}$.
Moreover there exists a Weil number $\beta$ of a QM-abelian variety over $\F_{\ell^f}$ such that $\Trd_{B_{\ell'}/F_{\ell'}} (R_{\ell'} (\Frob_k^m)) = \beta^m + \overline{\beta}^m$.
\end{lemma}

In the situation of the lemma above, we denote $\Trd_{B_{\ell'}/F_{\ell'}} (R_{\ell'} (\Frob_k^m))$ by $a_{A/k}(\Frob_k^m)$.

\begin{proof}
Fix a Frobenius $\Frob_k$.
Let $\Gamma$ be the closure of the subgroup generated by $\Frob_k$ and $N$ be the kernel of the restriction of $R_{\ell'}$ to the inertia group $I_k$ of $k$.
Then $\Gamma N$ is closed in $G_k$ and the corresponding field $k'$ is a one described in \cite{kj}*{Proposition 5.6}.
(See the proof of \cite{Jordan}*{Proposition 3.2}.)
Letting $A'$ be the Neron model of $A_{k'}$, the element $R_{\ell'}(\Frob_k)$ corresponds to $T_{\ell'}(\pi)$ under the canonical isomorphism $T_{\ell'} A \isom T_{\ell'} A'_{\F_k}$, where $\pi$ is the relative Frobenius endomorphism of $A'_{\F_k}/\F_k$.
Therefore the statement follows from \cite{kj}*{Proposition 4.20} and its proof.
\end{proof}

Also as is expected, the invariant $a_{A/k}(\Frob_k^m)$ is closely related to our Galois characters which we study above.
This relation helps us to derive the information about the rational points of $M^B$ from global behavior of our Galois characters.

\begin{lemma} \label{a_of_A/k_ver2}
Keep the same notation as in \cref{a_of_A/k}.
If $\ell \ne p$, then
\begin{equation*}
a_{A/k}(\Frob_k^m) \mod p_F = \rho_A(\Frob_k^m) + \ell^{mf} \rho_A(\Frob_k^m)^{-1}.
\end{equation*}
\end{lemma}

\begin{proof}
Since under the canonical map $T_p A \to A[p_B](\kbar)$ the automorphism $R_p(\Frob_k^m)$ corresponds to $\rho_A(\Frob_k^m)$, and since
\begin{equation*}
\begin{aligned}
\xymatrix{
O_{B,p_B}^* \ar[r] \ar[d]_{\Nrd \text{ or } \Trd} & \F_{p_B}^* \ar[d]^{N_{\F_{p_B}/\F_{p_F}} \text{ or } \tr_{\F_{p_B}/\F_{p_F}}} \\
O_{F,p_F}^* \ar[r] & \F_{p_F}^*
}
\end{aligned}
\end{equation*}
is commutative, we have that
\begin{equation*}
\begin{aligned}
a_{A/k}(\Frob_k^m) \mod p_F & = \tr_{\F_{p_B}/\F_{p_F}}(\rho_A(\Frob_k^m)), \\
\ell^{mf} \mod p_F & = N_{\F_{p_B}/\F_{p_F}}(\rho_A(\Frob_k^m)) \\
& = \rho_A(\Frob_k^m)^{1+p^r}
\end{aligned}
\end{equation*}
by \cref{a_of_A/k}.
Therefore
\begin{equation*}
\begin{aligned}
a_{A/k}(\Frob_k^m) \mod p_F & = \rho_A(\Frob_k^m) + \rho_A(\Frob_k^m)^{p^r} \\
& = \rho_A(\Frob_k^m) + \ell^{mf} \rho_A(\Frob_k^m)^{-1}.
\end{aligned}
\end{equation*}
\end{proof}

\begin{definition}
For positive integers $f$ and $e$ and for a rational prime $\ell$, let
\begin{equation*}
D(\ell^f, e) := \{ a^2, a^2 - \ell^{ef}, a^2 - 3\ell^{ef}, a^2-4\ell^{ef} | a \in C(\ell^f,e) \},
\end{equation*}
and $P(\ell^f, e)$ be the set of primes of $F$ dividing some nonzero element of $D(\ell^f,e)$ or dividing $n_{\lcm} \mathfrak{n}_F$.
For a positive integer $u$ and for a number field $k$ we also define
\begin{equation*}
Q(\ell^f, u) := \bigcup_{i = 1}^d P(\ell^f, u \frac{n_{\lcm}}{6}h'_ki).
\end{equation*}
Note that these are finite.
\end{definition}

For later use we need the following:

\begin{lemma} \label{Fbeta=Fsqrt-q}
Let $f$ and $e$ be positive integers and $\ell$ be a rational prime.
Let $\beta \in FR(\ell^f)$.
Assume that $\ell \nmid 2 \mathfrak{d}_{F/\Q}$, $f$ is odd, and $(\beta^e + \overline{\beta}^e)^2 = c \ell^{ef}$ for $c = 0,1,3$ or $4$.
Then there exists a nonnegative integer $a$ such that $\ell^a \mid n_{\lcm}$, $\beta = \zeta_{\ell^a}\sqrt{-\ell}^f$, and $F(\beta) = F(\sqrt{-\ell})$, where $\zeta_{\ell^a}$ is a primitive $\ell^a$-th root of unity.
\end{lemma}

\begin{proof}
By the equation $(\beta^e + \overline{\beta}^e)^2 = c \ell^{ef}$ for $c = 0,1,3$ or $4$, we have $\beta = \sqrt{\ell}^f \zeta$, where $\zeta$ is a $12e$-th roof of unity.
By the assumptions and by \cite{kj}*{Lemma 5.10}, the field $F(\beta)$ is totally imaginary and is quadratic over $F$, and the conjugate $\overline{\beta}$ is equal to $\sqrt{\ell}^f \zeta^{-1}$.
Since $F$ contains $\beta^2 + \overline{\beta}^2 = \ell^f (\zeta^2 + \zeta^{-2})$, the number $\zeta^2 + \zeta^{-2}$ lies in $F$, and hence $F(\zeta^2)$ is a quadratic extension of $F = F(\zeta^2 + \zeta^{-2})$, which yields that $\zeta^{2n_{\lcm}} = 1$ by the definition of $n_{\lcm}$.

Let $a$ be the $\ell$-adic value of the order of $\zeta$ and let $m = \ell^a$.
Since $\ell$ is odd by the assumption we have $m \mid n_{\lcm}$.
Now it is obvious that $\beta^m \in FR(\ell^{mf})$, and hence $[F(\beta^m):F] = 2$, i.e., $F(\beta) = F(\beta^m)$ since $mf$ is odd.
Thus in order to show the statement, it suffices to show that $\beta = \pm \sqrt{-\ell}^f$, i.e., $\zeta + \zeta^{-1} = 0$ under the assumption that $a=0$.

The field $F(\zeta)$ is contained in $F(\zeta, \sqrt{\ell}) = F(\beta, \sqrt{\ell})$, whose degree over $F$ divides $4$ since $\beta$ and $\sqrt{\ell}$ are quadratic over $F$.
Since now we are assuming that $\ell$ does not divides the order of $\zeta$, it is unramified in $F(\zeta)$.
Thus $F(\zeta)$ must not be equal to $F(\zeta, \sqrt{\ell})$, for, otherwise $F(\zeta)$ contains $\sqrt{\ell}$.
In particular $[F(\zeta) : F]$ divides $2$, and hence $\zeta + \zeta^{-1} \in F$.
Now $\beta + \overline{\beta}$, which lies in $F$, is equal to $\sqrt{\ell}^f (\zeta + \zeta^{-1})$.
Therefore if $\zeta + \zeta^{-1}$ were not zero, then since we are assuming that $f$ is odd, $\sqrt{\ell}$ would lie in $F$, which contradicts the assumption.
It shows the statement.
\end{proof}

Finally we recall some concepts and statements from descent theory.
Let $k$ be a field, $X$ a separated $k$-scheme of finite type, $A$ a commutative algebraic $k$-group, $f \colon Y \to X$ a right $X$-torsor under $A$, and let $\sigma$ be a 1-cocycle of $G_k$-module $A(\kbar)$.
Then we can define another right torsor $f^\sigma \colon Y^\sigma \to X$ under $A$, which satisfies $[Y^\sigma/X] = [Y/X] - [\sigma]$ in $H^1(X, A)$.
(See Example 2 of \cite{SkorobogatovTorsors}*{Lemma 2.2.3}.)

\begin{theorem}[\cite{SkorobogatovTorsors}*{Theorem 6.1.2 (a)}] \label{Br_lambda_obstruction}
Let $k$ be a number field, $X$ a proper $k$-scheme, and $S$ be a $k$-group of multiplicative type.
Let $f \colon Y \to X$ be a torsor under $S$ and $\lambda$ be the type of $f$ (\cite{SkorobogatovTorsors}*{Definition 2.3.2}).
Then if $Y^\sigma$ has no adelic points over $k$ for every $\sigma \in H^1(k, S)$, then $X(\A_k)^{\Br_\lambda } = \varnothing$, and in particular $X(\A_k)^{\Br_1} = X(\A_k)^{\Br} = \varnothing$.
(For the definition of $\Br_\lambda X$, see the introduction.)
\end{theorem}

This follows from \cite{SkorobogatovTorsors}*{(2) in the proof of Theorem 6.1.2 (a)}.

Now we are ready to show the main theorem.

\begin{theorem} \label{maintheorem1}
Let $k$ be a number field containing the normal closure of $F/\Q$.
Assume that there exist a prime $\mathfrak{q}$ of $k$ above a rational prime $q$ and a prime $p_F$ of $F$ above a rational prime $p$ satisfying all of the following conditions:
\begin{enumerate}
\item $\mathfrak{q} \nmid 2 \Delta$.
\item The absolute inertia degree of $\mathfrak{q}$ is odd.
\item $B \otimes_F F(\sqrt{-q}) \not\isom M_2(F(\sqrt{-q}))$.
\item $p_F \not\in Q(N(\mathfrak{q}), u)$, where $u$ is the integer as in \cref{varphi(Frob)^blahblah=q^blahblah}
\item For every prime $\mathfrak{p}$ of $k$ above $p$, the greatest common divisor of $2r$ and the absolute inertia degree of $\mathfrak{p}$ is $r$, where $r$ is the absolute inertia degree of $p_F$.
\item $p_F \mid \mathfrak{d}_{B/F}$.
\end{enumerate}
Then $M^B(\A_k)^{\Br_1} = \varnothing$ and in particular $M^B(k) = \varnothing$.
\end{theorem}

\begin{proof}
Let $\mathfrak{q}$ and $p_F$ be primes described in the assumptions and $q$ and $p$ be the residual characteristic of $\mathfrak{q}$ and $p_F$ respectively.
Let $r$ be the absolute inertia degree of $p_F$ and $p_B$ be the unique maximal ideal of $O_{B,p_F}$.
We may assume that $M^B$ has an adelic point $(x_v)_v$ since otherwise there is nothing to prove.
For each place $v$, the character $\varphi_{x_v}$ of $G_{k_v}$ is by definition the cohomology class represented by the $k_v$-torsor $X_{p_F,x_v}$ under the constant group $N_{\F_{p_B}/\F_{p_F}}^{-1}(\F_p^*)^{n_{\lcm}}$.
Also for every $\sigma \in H^1(k, N_{\F_{p_B}/\F_{p_F}}^{-1}(\F_p^*)^{n_{\lcm}})$, we have that $(X_{p_F,x_v})^{\sigma_{k_v}} \isom (X_{p_F}^\sigma)_{x_v}$.
Thus for a character $\varphi \colon G_k \to N_{\F_{p_B}/\F_{p_F}}^{-1}(\F_p^*)^{n_{\lcm}}$ we obtain that it satisfies $\varphi|_{G_{k_v}} = \varphi_{x_v}$ for every $v$ if and only if the point $(x_v)_v$ lifts to a $k$-adelic point of the $M^B$-torsor $X_{p_F}^\varphi$.
Hence the non-existence of such characters leads the desired result of the statement by \cref{Br_lambda_obstruction}.

Suppose for a contradiction that $\varphi \colon G_k \to N_{\F_{p_B}/\F_{p_F}}^{-1}(\F_p^*)^{n_{\lcm}}$ is such a character.
Let $K$ be a field extension of $k$ as in the proof of \cref{varphi(Frob)^blahblah=q^blahblah}.
Let $\mathfrak{Q}$ be a (unique) prime of $K$ above $\mathfrak{q}$ and $f$ be the inertia degree of $\mathfrak{q}$, which is equal to the one of $\mathfrak{Q}$ and is odd by the assumption.
Then by \cite{kj}*{Theorem 4.24} there exists $X_\mathfrak{Q} \in \mathscr{M}^B(K_\mathfrak{Q})$ representing a rational point $x_\mathfrak{q} \times K_\mathfrak{Q} \in M^B(K_\mathfrak{Q})$.
Letting $A_\mathfrak{Q}/K_\mathfrak{Q}$ be the underlying QM-abelian variety of $X_\mathfrak{Q}$, $m := u \frac{n_{\lcm}}{6} h'_k r$ (, which, note that, is even), and $a := a_{A_\mathfrak{Q}}(\Frob_\mathfrak{Q}^m)$, by \cref{a_of_A/k} we have $a \in C(q^f, m)$ and by \cref{a_of_A/k_ver2} we have
\begin{equation*}
a \mod p_F = \rho_{X_\mathfrak{Q}}(\Frob_\mathfrak{Q}^m) + q^{mf} \rho_{X_\mathfrak{Q}}(\Frob_\mathfrak{Q}^m)^{-1},
\end{equation*}
in particular $a^2 \mod p_F = (\varepsilon + \varepsilon^{-1})^2 q^{mf}$ for $\varepsilon := q^{-mf/2} \rho_{X_\mathfrak{Q}}(\Frob_\mathfrak{Q})^m \in \F_{p_B}^*$.
On the other hand, by \cref{rho^n=varphi} we have that $\varphi|_{G_{K_\mathfrak{Q}}} = \rho_{X_\mathfrak{Q}}^{n_{\lcm}}$ and hence by \cref{varphi(Frob)^blahblah=q^blahblah} we have $\varepsilon^{12} = 1$, in particular $(\varepsilon + \varepsilon^{-1})^2 = 0, 1, 3,$ or $4$.
Therefore we obtain that $a^2 \mod p_F = i q^{mf} \mod p$ for $i = 0, 1, 3,$ or $4$.
Now since $a \in C(q^f, m)$ and since $p_F \not\in P(q^f, m)$ by the assumptions, it follows that $a^2 = i q^{mf}$ for $i = 0, 1, 3,$ or $4$.

Let $\beta$ be a Weil number of a QM-abelian variety over $\F_{q^f}$ such that $a = \beta^m + \overline{\beta}^m$, which does exist by \cref{a_of_A/k}.
Then by \cref{Fbeta=Fsqrt-q} we have that $F(\beta) = F(\sqrt{-q})$.
Therefore since now we are assuming that $q \nmid \Delta'$ we obtain by \cite{kj}*{Proposition 4.16, 1} that $F(\sqrt{-q})$ splits $B/F$, which contradicts our assumptions, and thus we finish the proof.
(Note that the field $F'$ of \cite{kj}*{Proposition 4.16, 1} is the compositum of $F$ and the center of the endomorphism algebra of the QM-abelian variety, which is generated by a Weil number over $\Q$ in the case of finite fields.)
\end{proof}

\begin{remark}
As in the case of $M^B_{K_{p_F}} \to M^B$, we can define the ``Shimura covering" $M^B_2(p_F)$ of the canonical map $M^B_1(p_F) \to M^B_0(p_F)$ for $p_F \nmid \mathfrak{d}_{B/F} \mathfrak{D}_{F/\Q}$, and we 
can develop a similar argument about it.
From that, we can show $M^B_0(p_F)(\A_k)^{\Br_1} = \varnothing$, which is a generalization of the main theorem of \cite{kj}.
\end{remark}

\section{CM theory.} \label{Section:examples}

For a number field $k$ and for a positive integer $u$, we define $R(u)$ as the set of pairs $(\mathfrak{q}, p_F)$, where $\mathfrak{q}$ is a prime of $k$ and $p_F$ is a prime of $F$ above a rational prime $p$ satisfying all of the following conditions:
\begin{enumerate}
\item $\mathfrak{q} \nmid 2\mathfrak{d}_{F/\Q}$.
\item The absolute inertia degree of $\mathfrak{q}$ is odd.
\item $p_F \not\in Q(N(\mathfrak{q}), u)$.
\item For every prime $\mathfrak{p}$ of $k$ above $p$, the greatest common divisor of $2r$ and the absolute inertia degree of $\mathfrak{p}$ is $r$, where $r$ is the absolute inertia degree of $p_F$.
\end{enumerate}
This is a big set in the following sence:
there exists a positive density of $\mathfrak{q}$ satisfying the first two conditions above, and for each such a prime $\mathfrak{q}$, there exists a positive density of $p_F$ so that $(\mathfrak{q},p_F) \in R(u)$ (at least if $F$ and $k$ are Galois over the rationals).

For an element $(\mathfrak{q},p_F) \in R(2)$ and for the rational prime $q$ below $\mathfrak{q}$, if we choose a totally indefinite quaternion algebra $B$ over $F$ so that $p_F \mid \mathfrak{d}_{B/F}$, $B \otimes_F F(\sqrt{-q}) \not\isom M_2(F(\sqrt{-q}))$, and that $\mathfrak{q} \nmid N_{F/\Q}(\mathfrak{d}_{B/F})$, then by \cref{maintheorem1} we have that $M^B(k) = \varnothing$.
The condition $B \otimes_F F(\sqrt{-q}) \not\isom M_2(F(\sqrt{-q}))$ is obviously equivalent to that there exists a prime $\ell_F \mid \mathfrak{d}_{B/F}$ of $F$ which splits in $F(\sqrt{-q})$.

On the other hand, let $E$ be a totally imaginary quadratic extension of $F$, $\Phi$ a CM type of the CM-field $E$, $E^r,$ the reflex field of $(E, \Phi)$, and let $H_{E^r}$ be the Hilbert class field of $E^r$.
Then there exists a polarized abelian variety $X$ over $H_{E^r}$ which has CM by $(E,\Phi)$.
Thus if $B \otimes_F E \isom M_2(E)$, then an order of $B$ acts on $A := X^2/H_{E^r}$.
(Moreover by the definition of the action, we obtain that $E$ is contained in $\End_B^0(A)$, and hence $\End_B^0(A) = E$ by \cite{kj}*{Proposition 4.16}.)
Therefore if $A$ has a nice (with respect to our choice of $\Lambda$ and $\psi$, see \cref{Section:Shimura_covering}) polarization then $M^B(H_{E^r})$ is not empty.
(Recall that $M^B$ is the coarse moduli scheme of $\mathscr{M}^B \isom \mathscr{M}^{B,\text{rat}}$.)
\cref{maintheorem1} (of course!) does not contradicts this situation.

\begin{lemma}
Let $E$ be a totally imaginary quadratic extension of $F$, $\Phi$ a CM-type of the CM-field $E$, $E^r$ the reflex field of $(E, \Phi)$, and let $H_{E^r}$ be the Hilbert class field of $E^r$.
Assume that $B \otimes_F E \isom M_2(E)$.
Let $k$ be a number field containing $E$ and $H_{E^r}$.
Then for any $(\mathfrak{q}, p_F) \in R(1)$, if $p_F \mid \mathfrak{d}_{B/F}$, then $B \otimes_F F(\sqrt{-q}) \isom M_2(F(\sqrt{-q}))$.
\end{lemma}

\begin{proof}
Let $(\mathfrak{q}, p_F) \in R(1)$ such that $p_F \mid \mathfrak{d}_{B/F}$.
We show that $B \otimes_F F(\sqrt{-q}) \isom M_2(F(\sqrt{-q}))$.
First we claim that there exists $\beta \in FR(N(\mathfrak{q}))$ such that $F(\beta) = E$.
There exists a QM-abelian variety $A'$ over $H_{E^r}$ such that $\End_B^0 A' = E$, and in particular so does over $k$.
Thus by \cite{kj}*{Proposition 5.6} there exists a QM-abelian variety $A$ over $\F_\mathfrak{q}$ such that $\End_B^0 A$ contains $E$.
Let $\beta \in  FR(N(\mathfrak{q}))$ be one of its Weil number.
Then by the assumption that $\mathfrak{q}$ is unramified over the rationals and that its absolute inertia index is odd and by \cite{kj}*{Lemma 5.10}, we have that $\beta$ is totally imaginary element and $[F(\beta):F]=2$.
Hence by \cite{kj}*{Proposition 4.16} (and by its proof) $\End_B^0(A/\F_\mathfrak{q}) = F(\beta)$, which proves the claim.
It follows that $p_F$ ramifies in $F(\beta)$ by the assumption that $B \otimes_F E \isom M_2(E)$ and that every prime of $k$ above $p_F$ has odd inertia degree over $F$.
Thus $p_F$ divides $(\beta + \overline{\beta})^2 - 4N(\mathfrak{q}) = (\beta - \overline{\beta})^2$, the discriminant of the order $O_F[\beta]$ over $O_F$ (see, for example, \cite{SerreLF}*{Chapter III, Section 2, Corollary}), and hence divides $(\beta^e + \overline{\beta}^e)^2 - 4N(\mathfrak{q})^e = (\beta^e - \overline{\beta}^e)^2$ for every positive integer $e$.
Now by the assumption $p_F$ does not belong to $Q(N(\mathfrak{q}),1)$.
Therefore the algebraic integer $(\beta^e + \overline{\beta}^e)^2 - 4N(\mathfrak{q})^e$ must be zero for a positive integer $e$, which yields that $F(\beta) = F(\sqrt{-q})$ by \cref{Fbeta=Fsqrt-q}.
Consequently, $F(\sqrt{-q}) = E$ splits $B$, which was what we wanted.
\end{proof}

\begin{acknowledgements}
I would like to thank my supervisor Takeshi Saito for many helpful suggestions and corrections.
This work was supported by JSPS KAKENHI Grant Number JP23KJ0568.
\end{acknowledgements}


\begin{bibdiv}
\begin{biblist}

\bib{AraiNonexistence}{article}{
   author={Arai, Keisuke},
   title={Non-existence of points rational over number fields on Shimura
   curves},
   journal={Acta Arith.},
   volume={172},
   date={2016},
   number={3},
   pages={243--250},
   issn={0065-1036},
   review={\MR{3460813}},
   doi={10.4064/aa8071-10-2015},
}
\bib{AraiBM}{article}{
   author={Arai, Keisuke},
   title={Rational points on Shimura curves and the Manin obstruction},
   journal={Nagoya Math. J.},
   volume={230},
   date={2018},
   pages={144--159},
   issn={0027-7630},
   review={\MR{3798622}},
   doi={10.1017/nmj.2017.6},
}
\bib{BreenLabesse}{book}{
   author={Boutot, Jean-Fran\c{c}ois},
   author={Breen, Lawrence},
   author={G\'{e}rardin, Paul},
   author={Giraud, Jean},
   author={Labesse, Jean-Pierre},
   author={Milne, James Stuart},
   author={Soul\'{e}, Christophe},
   title={Vari\'{e}t\'{e}s de Shimura et fonctions $L$},
   language={French},
   series={Publications Math\'{e}matiques de l'Universit\'{e} Paris VII
   [Mathematical Publications of the University of Paris VII]},
   volume={6},
   publisher={Universit\'{e} de Paris VII, U.E.R. de Math\'{e}matiques, Paris},
   date={1979},
   pages={178},
   review={\MR{680404}},
}
\bib{DeligneSV}{article}{
   author={Deligne, Pierre},
   title={Vari\'{e}t\'{e}s de Shimura: interpr\'{e}tation modulaire, et techniques de
   construction de mod\`eles canoniques},
   language={French},
   conference={
      title={Automorphic forms, representations and $L$-functions},
      address={Proc. Sympos. Pure Math., Oregon State Univ., Corvallis,
      Ore.},
      date={1977},
   },
   book={
      series={Proc. Sympos. Pure Math., XXXIII},
      publisher={Amer. Math. Soc., Providence, R.I.},
   },
   date={1979},
   pages={247--289},
   review={\MR{546620}},
}
\bib{dVP}{article}{
   author={de Vera-Piquero, Carlos},
   title={The Shimura covering of a Shimura curve: automorphisms and
   \'{e}tale subcoverings},
   journal={J. Number Theory},
   volume={133},
   date={2013},
   number={10},
   pages={3500--3516},
   issn={0022-314X},
   review={\MR{3071825}},
   doi={10.1016/j.jnt.2013.04.018},
}
\bib{FaltingsChai}{book}{
   author={Faltings, Gerd},
   author={Chai, Ching-Li},
   title={Degeneration of abelian varieties},
   series={Ergebnisse der Mathematik und ihrer Grenzgebiete (3) [Results in
   Mathematics and Related Areas (3)]},
   volume={22},
   note={With an appendix by David Mumford},
   publisher={Springer-Verlag, Berlin},
   date={1990},
   pages={xii+316},
   isbn={3-540-52015-5},
   review={\MR{1083353}},
   doi={10.1007/978-3-662-02632-8},
}
\bib{Fu}{book}{
   author={Fu, Lei},
   title={Etale cohomology theory},
   series={Nankai Tracts in Mathematics},
   volume={14},
   edition={Revised edition},
   publisher={World Scientific Publishing Co. Pte. Ltd., Hackensack, NJ},
   date={2015},
   pages={x+611},
   isbn={978-981-4675-08-6},
   review={\MR{3380806}},
   doi={10.1142/9569},
}
\bib{Jordan}{article}{
   author={Jordan, Bruce W.},
   title={Points on Shimura curves rational over number fields},
   journal={J. Reine Angew. Math.},
   volume={371},
   date={1986},
   pages={92--114},
   issn={0075-4102},
   review={\MR{859321}},
   doi={10.1515/crll.1986.371.92},
}
\bib{Kottwitz}{article}{
   author={Kottwitz, Robert E.},
   title={Points on some Shimura varieties over finite fields},
   journal={J. Amer. Math. Soc.},
   volume={5},
   date={1992},
   number={2},
   pages={373--444},
   issn={0894-0347},
   review={\MR{1124982}},
   doi={10.2307/2152772},
}
\bib{Lan}{book}{
   author={Lan, Kai-Wen},
   title={Arithmetic compactifications of PEL-type Shimura varieties},
   series={London Mathematical Society Monographs Series},
   volume={36},
   publisher={Princeton University Press, Princeton, NJ},
   date={2013},
   pages={xxvi+561},
   isbn={978-0-691-15654-5},
   review={\MR{3186092}},
   doi={10.1515/9781400846016},
}
\bib{kj}{article}{
title={Rational points on Shimura varieties classifying abelian varieties with quaternionic multiplication}, 
author={Matsuda, K.},
year={2023},
eprint={2311.10668}
}
\bib{MazurX1}{article}{
   author={Mazur, B.},
   title={Modular curves and the Eisenstein ideal},
   note={With an appendix by Mazur and M. Rapoport},
   journal={Inst. Hautes \'{E}tudes Sci. Publ. Math.},
   number={47},
   date={1977},
   pages={33--186 (1978)},
   issn={0073-8301},
   review={\MR{488287}},
}
\bib{Milne}{article}{
   author={Milne, J. S.},
   title={Points on Shimura varieties mod $p$},
   conference={
      title={Automorphic forms, representations and $L$-functions},
      address={Proc. Sympos. Pure Math., Oregon State Univ., Corvallis,
      Ore.},
      date={1977},
   },
   book={
      series={Proc. Sympos. Pure Math., XXXIII},
      publisher={Amer. Math. Soc., Providence, R.I.},
   },
   date={1979},
   pages={165--184},
   review={\MR{546616}},
}
\bib{MilneSh}{article}{
   author={Milne, J. S.},
   title={Introduction to Shimura varieties},
   conference={
      title={Harmonic analysis, the trace formula, and Shimura varieties},
   },
   book={
      series={Clay Math. Proc.},
      volume={4},
      publisher={Amer. Math. Soc., Providence, RI},
   },
   date={2005},
   pages={265--378},
   review={\MR{2192012}},
}
\bib{MumfordAV}{book}{
   author={Mumford, David},
   title={Abelian varieties},
   series={Tata Institute of Fundamental Research Studies in Mathematics},
   volume={5},
   publisher={Published for the Tata Institute of Fundamental Research,
   Bombay by Oxford University Press, London},
   date={1970},
   pages={viii+242},
   review={\MR{0282985}},
}
\bib{Neukirch}{book}{
   author={Neukirch, J\"{u}rgen},
   title={Algebraic number theory},
   series={Grundlehren der mathematischen Wissenschaften [Fundamental
   Principles of Mathematical Sciences]},
   volume={322},
   note={Translated from the 1992 German original and with a note by Norbert
   Schappacher;
   With a foreword by G. Harder},
   publisher={Springer-Verlag, Berlin},
   date={1999},
   pages={xviii+571},
   isbn={3-540-65399-6},
   review={\MR{1697859}},
   doi={10.1007/978-3-662-03983-0},
}
\bib{RdVP}{article}{
   author={Rotger, Victor},
   author={de Vera-Piquero, Carlos},
   title={Galois representations over fields of moduli and rational points
   on Shimura curves},
   journal={Canad. J. Math.},
   volume={66},
   date={2014},
   number={5},
   pages={1167--1200},
   issn={0008-414X},
   review={\MR{3251768}},
   doi={10.4153/CJM-2013-020-3},
}
\bib{SerreLF}{book}{
   author={Serre, Jean-Pierre},
   title={Local fields},
   series={Graduate Texts in Mathematics},
   volume={67},
   note={Translated from the French by Marvin Jay Greenberg},
   publisher={Springer-Verlag, New York-Berlin},
   date={1979},
   pages={viii+241},
   isbn={0-387-90424-7},
   review={\MR{0554237}},
}
\bib{Shimura}{article}{
   author={Shimura, Goro},
   title={On the real points of an arithmetic quotient of a bounded
   symmetric domain},
   journal={Math. Ann.},
   volume={215},
   date={1975},
   pages={135--164},
   issn={0025-5831},
   review={\MR{572971}},
   doi={10.1007/BF01432692},
}
\bib{SkorobogatovTorsors}{book}{
   author={Skorobogatov, Alexei},
   title={Torsors and rational points},
   series={Cambridge Tracts in Mathematics},
   volume={144},
   publisher={Cambridge University Press, Cambridge},
   date={2001},
   pages={viii+187},
   isbn={0-521-80237-7},
   review={\MR{1845760}},
   doi={10.1017/CBO9780511549588},
}
\bib{SkorobogatovShimuraCovering}{article}{
   author={Skorobogatov, Alexei},
   title={Shimura coverings of Shimura curves and the Manin obstruction},
   journal={Math. Res. Lett.},
   volume={12},
   date={2005},
   number={5-6},
   pages={779--788},
   issn={1073-2780},
   review={\MR{2189238}},
   doi={10.4310/MRL.2005.v12.n5.a14},
}
\bib{Voight}{book}{
   author={Voight, John},
   title={Quaternion algebras},
   series={Graduate Texts in Mathematics},
   volume={288},
   publisher={Springer, Cham},
   date={2021},
   pages={xxiii+885},
   isbn={978-3-030-56692-0},
   isbn={978-3-030-56694-4},
   review={\MR{4279905}},
   doi={10.1007/978-3-030-56694-4},
}
\bib{Zink}{article}{
   author={Zink, Thomas},
   title={Isogenieklassen von Punkten von Shimuramannigfaltigkeiten mit
   Werten in einem endlichen K\"{o}rper},
   language={German},
   journal={Math. Nachr.},
   volume={112},
   date={1983},
   pages={103--124},
   issn={0025-584X},
   review={\MR{726854}},
   doi={10.1002/mana.19831120106},
}





\end{biblist}
\end{bibdiv}
\end{document}